\documentclass[11pt]{amsart}
\usepackage{amscd,amsmath,amssymb,amsfonts,verbatim}
\usepackage[cmtip, all]{xy}

\newtheorem{thm}{Theorem}[section]
\newtheorem{prop}[thm]{Proposition}
\newtheorem{lem}[thm]{Lemma}
\newtheorem{cor}[thm]{Corollary}

\theoremstyle{definition}

\newtheorem{defn}[thm]{Definition}
\theoremstyle{remark}
\newtheorem{remk}[thm]{Remark}
\newtheorem{remks}[thm]{Remarks}

\newtheorem{exm}[thm]{Example}
\newtheorem{exms}[thm]{Examples}
\newtheorem{notat}[thm]{Notation}
\numberwithin{equation}{section}

{\hfill$\square$\end{defn}}
{\hfill$\square$\end{remk}}
{\hfill$\square$\end{remks}}
{\hfill$\square$\end{exm}}
{\hfill$\square$\end{exms}}
{\hfill$\square$\end{notat}}

\newcommand{\thmref}{Theorem~\ref}
\newcommand{\propref}{Proposition~\ref}
\newcommand{\corref}{Corollary~\ref}

\newcommand{\lemref}{Lemma~\ref}

\newcommand{\sA}{{\mathcal A}}

\newcommand{\sC}{{\mathcal C}}
\newcommand{\sD}{{\mathcal D}}
\newcommand{\sE}{{\mathcal E}}
\newcommand{\sF}{{\mathcal F}}
\newcommand{\sG}{{\mathcal G}}
\newcommand{\sH}{{\mathcal H}}
\newcommand{\sI}{{\mathcal I}}

\newcommand{\sK}{{\mathcal K}}

\newcommand{\sO}{{\mathcal O}}
\newcommand{\sP}{{\mathcal P}}

\newcommand{\sT}{{\mathcal T}}
\newcommand{\sU}{{\mathcal U}}

\newcommand{\sW}{{\mathcal W}}
\newcommand{\sX}{{\mathcal X}}
\newcommand{\sY}{{\mathcal Y}}
\newcommand{\sZ}{{\mathcal Z}}
\newcommand{\A}{{\mathbb A}}

\newcommand{\K}{{\mathbb K}}

\newcommand{\N}{{\mathbb N}}
\renewcommand{\P}{{\mathbb P}}

\newcommand{\Z}{{\mathbb Z}}

\newcommand{\surj}{\twoheadrightarrow}
\newcommand{\inj}{\hookrightarrow}
\newcommand{\red}{{\rm red}}

\newcommand{\Hom}{{\rm Hom}}

\newcommand{\Spec}{{\rm Spec \,}}

\newcommand{\Sch}{{\operatorname{\mathbf{Sch}}}}

\newcommand{\Tot}{{\operatorname{\rm Tot}}}

\newcommand{\Nis}{{\operatorname{Nis}}}

\newcommand{\ds}{{/\kern-3pt/}}

\renewcommand{\red}{{\rm red}}

\newcommand{\Proj}{{\operatorname{Proj}}}

\newcommand{\Mod}{{\mathbf{Mod}}}

\newcommand{\ov}{\overline}

\renewcommand{\dim}{\text{\rm dim}}

\newcommand{\tuborg}{\left\{\begin{array}{ll}}
\newcommand{\sluttuborg}{\end{array}\right.}

\newcommand{\To}{\longrightarrow}

\begin{document}
\title{Algebraic $K$-theory of quotient stacks}
\author{Amalendu Krishna, Charanya Ravi}
\address{School of Mathematics, Tata Institute of Fundamental Research,  
1 Homi Bhabha Road, Colaba, Mumbai, India}
\email{amal@math.tifr.res.in}

\address{Department of Mathematics, University of Oslo, P.O. Box 1053 Blindern, 
N-0316 Oslo, Norway}
\email{charanyr@math.uio.no}


\baselineskip13.9pt


\begin{abstract}
We prove some fundamental results like localization, excision, 
Nisnevich descent, and the regular blow-up formula for the 
algebraic $K$-theory of certain stack quotients of schemes with
affine group scheme actions. We show that the homotopy $K$-theory of
such stacks is homotopy invariant. This implies a similar
homotopy invariance property of the algebraic $K$-theory with  
coefficients. 
\end{abstract}

\maketitle
\section{Introduction}\label{sec:Intro}
The higher algebraic $K$-theories of Quillen and Thomason--Trobaugh
are among the most celebrated discoveries in Mathematics.
Fundamental results like localization, excision, 
Nisnevich descent, and the blow-up formula have played pivotal roles in 
almost every recent breakthrough in the $K$-theory of schemes 
(e.g., see \cite{Corti}, \cite{CHSW} and \cite{Schl-1}).   
The generalization of these results
to equivariant $K$-theory is the theme of this paper.

The significance of equivariant $K$-theory \cite{Thom}
in the study of the ordinary (non-equivariant) $K$-theory is essentially based on two principles. 
First, it often turns out 
that the presence of a group action allows one to exploit 
representation-theoretic tools to study equivariant $K$-theory. Second, 
there are results (see, for instance, \cite[Theorem~32]{Merk})
which directly connect equivariant algebraic $K$-theory with the ordinary 
$K$-theory of schemes with group action.     
These principles have been effectively used in the past to study both
equivariant and ordinary algebraic $K$-theory
(see, for instance, \cite{AR}, \cite{VV}).
In addition, equivariant $K$-theory often allows one to
understand various other cohomology theories of moduli stacks and moduli spaces from 
the $K$-theoretic point of view.

However, any serious progress 
towards the applicability of equivariant $K$-theory (of vector
bundles) requires analogues for quotient stacks of the
fundamental results of Thomason--Trobaugh. 
The goal of this paper is to establish these results so that a 
very crucial gap in the study of the $K$-theory of quotient stacks could be 
filled.
Special cases of these results were earlier 
proven in \cite{AK}, \cite{KO} and \cite{HKO}.
Here is a summary our main results.
The precise statements and the underlying notation
can be found in the body of the text. We fix a field $k$.

\begin{thm}\label{thm:Main-1}
Let $\sX$ be a nice quotient stack over $k$ with resolution property. 
Let $\K$ denote the (non-connective)  
$K$-theory presheaf on the 2-category of nice quotient stacks. 
Let $\sZ \inj \sX$ be a closed substack with open complement $\sU$. 
Then the following hold.
\begin{enumerate}
\item
There is a homotopy fibration sequence of $S^1$-spectra
\[
\K(\sX ~ {\rm on} ~ \sZ) \to \K(\sX) \to \K(\sU).
\]
\item
The presheaf $\sX \mapsto \K(\sX)$ satisfies excision.
\item
The presheaf $\sX \mapsto \K(\sX)$ satisfies Nisnevich descent.
\item
The presheaf $\sX \mapsto \K(\sX)$ satisfies descent  
for regular blow-ups.
\end{enumerate}
\end{thm}

\begin{thm}\label{thm:Main-2}
The non-connective homotopy $K$-theory presheaf
$KH$ on the 2-category of nice quotient stacks with resolution property
satisfies the following.
\begin{enumerate}
\item
It is invariant under every vector bundle morphism (Thom isomorphism for
stacks).
\item
It satisfies localization, excision, Nisnevich descent and
descent for regular blow-ups.
\item
If $\sX$ is the stack quotient of a scheme by a finite nice group, 
then $KH(\sX)$ is invariant under infinitesimal extensions.
\end{enumerate}
\end{thm}

The following result shows that $K$-theory with
coefficients for quotient stacks is homotopy invariant, i.e., it
satisfies the Thom isomorphism.
No case of this result was yet known for stacks which are not schemes.

\begin{thm}\label{thm:Mian-3}
Let $\sX$ be a nice quotient stack over $k$ with resolution property
and let $f: \sE \to \sX$ be a vector bundle. Then the following hold.
\begin{enumerate}
\item
For any integer $n$ invertible in $k$, the map
$f^*: \K(\sX ; {\Z}/n) \to \K(\sE ; {\Z}/n)$ is a homotopy equivalence.
\item
For any integer $n$ nilpotent in $k$, the map
$f^*: \K(\sX ; {\Z}[{1}/n]) \to \K(\sE ; {\Z}[1/n])$ is a homotopy equivalence.
\end{enumerate}
\end{thm}

In the above results, a nice quotient stack means a stack of the form $[X/G]$, 
where $G$ is an affine group scheme over $k$ acting on a $k$-scheme $X$
such that $G$ is nice, i.e., it is either linearly reductive over $k$ or ${\rm char}(k) = 0$.
Group schemes of multiplicative type (e.g., diagonalizable group 
schemes) are notable examples of this in positive characteristic.
We refer to ~\ref{sec:LRGS} for more details.

{\bf Applications:} 
Similar to the case of schemes, one expects the above results to be of
central importance in the study of the $K$-theory of quotient stacks.
Already by now, there have been two immediate major
applications: (1) the $cdh$-descent and,
(2) Weibel's conjecture for negative $KH$-theory of stacks. 
In a sense, these applications motivated the results of this paper.

Hoyois \cite{Hoyois-2} has constructed a variant of $KH$-theory for nice 
quotient
stacks and has used the main results of this paper to prove the $cdh$-descent for
this variant. The results of this paper 
(and their generalizations) have also been
recently used by Hoyois and the first author \cite{HK} to prove
$cdh$-descent for the $KH$-theory 
(as defined in \S~\ref{sec:Finite}) of nice stacks, 
and to prove Weibel's conjecture for
the vanishing of negative $KH$-theory of such stacks.  

Another application of the above results is related to a rigidity type
theorem for the $K$-theory of semi-local rings. Let $A$ be a normal
semi-local ring with isolated singularity with an action of a finite group $G$
and let $\widehat{A}$ denote its completion along the Jacobson radical.
Rigidity question asks if the map $K'_*(G,A) \to K'_*(G, \widehat{A})$ is 
injective. If $G$ is trivial, this was proven for $K'_0(G,A)$ by Kamoi and 
Kurano \cite{KK} for certain type of isolated singularities. 
They apply this result to characterize certain semi-local rings.
The main tool of \cite{KK} is \thmref{thm:Main-1} 
for the ordinary $K$-theory of singular rings. 
We hope that the localization theorem for the $K$-theory of
quotient stacks can now be used to prove the equivariant version of this
rigidity theorem.

\section{Perfect complexes on quotient stacks}\label{sec:PCS}
Throughout this text, we work over a fixed base field $k$ of arbitrary 
characteristic. In this section, we fix notations, recall basic definitions
and prove some preliminary results. We conclude the section with the proof of 
an excision property for the derived category of perfect complexes on stacks. 

\subsection{Notations and definitions}
We let $\Sch_k$ denote the category of separated schemes of finite type
over $k$. A {\sl scheme} in this paper will mean an object
of $\Sch_k$. 
A group scheme $G$ will mean an affine group scheme over $k$.
Recall that a stack $\sX$ (of finite type) over the big fppf site of $k$
is said to be an algebraic stack over $k$
if the diagonal of $\sX$ is representable by algebraic spaces and $\sX$ admits  
a smooth, representable and surjective morphism $U \to \sX$ 
from a scheme $U$. 
Throughout this text a `stack' will always refer to an algebraic stack.
We shall say that $\sX$ is a {\sl quotient stack} if it is
a stack of the form $[X/G]$ (see, for instance, \cite[\S~2.4.2]{LMB}), 
where $G$ is an affine group scheme acting on a scheme $X$.

\subsection{Nice stacks}\label{sec:LRGS}
Given a group scheme $G$, let $\Mod^G(k)$ denote the
category of $k$-modules with $G$-action. Recall that $G$ is said to be
linearly reductive if the ``functor of $G$-invariants'' 
$(-)^G: \Mod^G(k) \to \Mod(k)$, given by the submodule of $G$-invariant 
elements, is exact. If ${\rm char}(k) = 0$, it is well known that
$G$ is linearly reductive if and only if it is reductive. 
In general, it follows from \cite[Propositions~2.5, 2.7, Theorem~2.16]{AOV}
that $G$ is linearly reductive if there is an extension
\begin{equation}\label{eqn:Lin-red}
1 \to G_1 \to G \to G_2 \to 1,
\end{equation}
where each of $G_1$ and $G_2$ is either finite over $k$ of degree prime
to the exponential characteristic of $k$ or, is of multiplicative type
({\'e}tale locally diagonalizable) over $k$.
One knows that linearly reductive group schemes in positive characteristic
are closed under the operations of taking closed subgroups and base change.

\begin{defn}\label{defn:Nice}
We shall say that a group scheme $G$ is {\sl nice}, if
either it is linearly reductive or ${\rm char}(k) = 0$.
If $G$ is nice and it acts on a scheme $X$, we shall say that the
resulting quotient stack $[X/G]$ is nice.
\end{defn}

\subsection{Perfect complexes on stacks}\label{sec:sheaves}
Given a stack $\sX$, let $Sh(\sX)$ (resp. $Mod(\sX)$, resp. $QC(\sX)$) 
denote the abelian category of sheaves of abelian groups 
(resp. sheaves of $\sO_{\sX}$-modules, resp. quasi-coherent sheaves) on 
the smooth-{\'e}tale site Lis-Et($\sX$) of $\sX$.
Let $Ch_{qc}(\sX)$ (resp. $Ch(QC(\sX))$) denote the 
category of all (possibly unbounded) chain complexes over $Mod(\sX)$ 
whose cohomology lie in $QC(\sX)$ (resp. the category of all chain complexes 
over $QC(\sX)$). 
Let $D_{qc}(\sX)$ (resp. $D(QC(\sX))$) denote the corresponding 
derived category. Let $D(\sX)$ denote the unbounded derived category of
$Mod(\sX)$.
If $\sZ \inj \sX$ is a closed substack with open complement $j: \sU \inj \sX$, 
we let $Ch_{qc,\sZ}(\sX) = \{ \sF\in Ch_{qc}(\sX) | j^*(\sF) \stackrel{q. \ iso.}
{\To} 0\}$. 
The derived category of $Ch_{qc,\sZ}(\sX)$ will be denoted by $D_{qc,\sZ}(\sX)$.
Recall that a stack $\sX$ is said to have the resolution property if
every coherent sheaf on $\sX$ is a quotient of a vector bundle.
 
\begin{lem}\label{lem:Thom-Res}
Let $\sX$ be the stack quotient of a scheme $X$ with an action of a group
scheme $G$. Then the following hold.
\begin{enumerate}
\item
Every quasi-coherent sheaf on $\sX$ is the direct limit of
its coherent subsheaves.
\item
$\sX$ has the resolution property if $X$ has an ample family of
$G$-equivariant line bundles.
In particular, $\sX$ has the resolution property if $X$ is 
normal with an ample family of (non-equivariant)
line bundles.
\item 
$\sX$ has the resolution property if $X$ is quasi-affine.
\end{enumerate}
\end{lem}
\begin{proof}
Part (1) is \cite[Lemma~1.4]{Thom1}.
For (2), note that $[\Spec(k)/G]$ has the resolution property
\cite[Lemma~2.4]{Thom1}. Therefore,
if $X$ has an ample family of
$G$-equivariant line bundles, it follows from \cite[Lemma~2.6]{Thom1}
that $\sX$ has the resolution property.
If $X$ is normal with an ample family of (non-equivariant) line bundles,
it follows from \cite[Lemmas~2.10, 2.14]{Thom1} that 
$\sX$ has the resolution property.
(3) is well known, and for example, follows from \cite[Lemma~7.1]{HR}.
\end{proof}

Recall from \cite[Definition~I.4.2]{SGA6} that a complex of
$\sO_X$-modules on a Noetherian scheme is perfect if it is Zariski locally 
quasi-isomorphic to a bounded complex of locally free sheaves.

\begin{defn} \label{defn:Perfect-complex}
Let $\sX$ be a stack over $k$.
A chain complex $P \in Ch_{qc}(\sX)$ is called {\it perfect} if for any 
affine scheme $U = \Spec(A)$ with a smooth morphism $s: U \to \sX$, 
the complex of $A$-modules $s^*(P) \in Ch(Mod(A))$ is quasi-isomorphic 
to a bounded complex of finitely-generated projective $A$-modules.
\end{defn}

We shall denote the category of perfect complexes on $\sX$ by
$\rm Perf(\sX)$ and its derived category by $D^{\rm perf}(\sX)$.
For a quotient stack with resolution property, we can characterize
perfect complexes in terms of their pull-backs to the total space
of the quotient map.

\begin{lem}\label{lem:Desc-perf}
Let $f:X' \to X$ be a faithfully flat map of Noetherian schemes.
Let $P$ be a chain-complex of
quasi-coherent sheaves on $X$ such that $f^*(P)$ is perfect on $X'$.
Then $P$ is a perfect complex on $X$.
\end{lem}
\begin{proof}
By \cite[Proposition~2.2.12]{TT}, a complex of quasi-coherent sheaves
is perfect if and only if it is cohomologically bounded above, its cohomology
sheaves are coherent, and it has locally finite Tor-amplitude. But all these
properties are known to descend from a faithfully flat cover.
\end{proof}

\begin{prop} \label{prop:perfect-complex-[X/G]}
Let $\sX$ be the stack quotient of a scheme $X$ with an action of a group
scheme $G$ and let $u: X \to \sX$ be the quotient map. 
Assume that $\sX$ has the resolution property.
Let $P$ be a chain complex of quasi-coherent $\sO_{\sX}$-modules.
Then the following are equivalent.
\begin{enumerate}
\item
$P$ is perfect.
\item
$u^*(P)$ is perfect.
\item
$u^*(P)$ is quasi-isomorphic to a bounded complex of 
$G$-equivariant vector bundles in $Ch(QC^G(X))$, where $QC^G(X)$ denotes the
category of $G$-equivariant quasi-coherent sheaves on $X$.
\end{enumerate}
\end{prop}
\begin{proof}
(1) $\implies$ (2). We let $Q = u^*(P)$.
Consider an open cover of $X$ by affine open subsets 
$\{ \Spec(A_i) \}$. Let $s: U \to [X/G]$ be an atlas 
and let $s_i: U_i \to \Spec (A_i)$ be its base change to $\Spec (A_i)$,
where $U_i$ are algebraic spaces. Take \'etale covers $t_i: V_i \to U_i$ of 
$U_i$, where $V_i$'s are schemes. Let $f_i: V_i \to U$ and 
$g_i: V_i \to \Spec (A_i)$ denote the 
obvious composite maps. It follows from (1) that 
${\bf L}g_i^*(Q|_{\Spec(A_i)}) \simeq {\bf L}f_i^* (s^*(P))$ 
is a perfect complex on $V_i$. Therefore by \lemref{lem:Desc-perf}, 
$Q|_{\Spec(A_i)}$ is a perfect complex in $Ch(Mod(A_i))$. Equivalently,
$Q$ is perfect.

(2) $\implies$ (3). 
We want to apply the Inductive Construction Lemma 1.9.5 of \cite{TT} with
$\sA = QC^G(X)$, $\sD =$ the category of $G$-equivariant vector
bundles on $X$ and 
$\sC =$ the category of complexes in $Ch(QC^G(X))$ satisfying (2).
It is enough to verify that hypothesis 1.9.5.1 of {\it loc. cit.} holds. 

Suppose $C \in \sC$ such that $H^i(C) = 0$ for $i \geq n$, 
and $q: \sF \surj H^{n-1}(C)$ in $QC^G(X)$. 
By \cite[Proposition 2.2.3]{TT}, 
$\sG = H^{n-1}(C)$ is a coherent $\sO_X$-module. 
Therefore,  $\sG$ is a coherent $G$-module. 
By \lemref{lem:Thom-Res} (1), we can write $\sF = \varinjlim \sF_{\alpha}$,
where $\sF_{\alpha}$ are coherent $G$-submodules of $\sF$.
Under the forgetful functor, this gives an epimorphism
$q: \varinjlim \sF_{\alpha} \surj \sG$ in $QC(X)$,
where $\sF_{\alpha}, \ \sG$ are coherent modules.

Now, as $\sG$ is coherent and $X$ is Noetherian,
we can find an $\alpha$ such that 
the composite map $\sF_{\alpha} \inj \sF \xrightarrow{q} \sG$ is surjective. 
By the resolution property, there exists $\sE \in \sD$ such that
$\sE \surj \sF_{\alpha}$. 
Hence the composite $\sE \to \sF_{\alpha} \inj \sF \xrightarrow{q} \sG$ 
is also surjective. 
Applying the conclusion of \cite[Lemma 1.9.5]{TT} to $C^{\bullet} = P$ and 
$D^{\bullet} = 0$, we get a bounded above complex $E$ of $G$-vector bundles 
and a quasi-isomorphism $\phi: E \xrightarrow{\sim} P$ in $Ch(QC^G(X))$.  
Therefore, $E \in \sC$. 

Since $X$ is Noetherian, $E$ has globally finite Tor-amplitude.
To show that $Q$ is quasi-isomorphic to a bounded complex over $\sD$,
it suffices to prove that the good truncation $\tau^{\geq a}(E)$ is a bounded
complex of $G$-equivariant vector bundles
and the map $E \to \tau^{\geq a}(E)$ is a quasi-isomorphism.
It is enough to prove this claim by forgetting the $G$-action.
But this follows exactly
along the lines of the proof of \cite[Proposition 2.2.12]{TT}.
(3) $\implies$ (1) is clear.
\end{proof}

\subsection{Perfect complexes and compact objects of $D_{qc}(\sX)$}
\label{sec:compact}
Recall that if $\sT$ is a triangulated category  which is closed under
small coproducts, then an object $E \in {\rm Obj}(\sT)$ is called compact
if the functor $\Hom_{\sT}(E, -)$ on $\sT$ commutes with small
coproducts. The full triangulated subcategory of compact objects in $\sT$
is denoted by $\sT^{c}$.
If $X$ is a scheme, one of the main results of \cite{TT}
is that a chain complex $P \in Ch_{qc}(X)$ is perfect if and only if it
is a compact object of $D_{qc}(X)$. For quotient stacks, this is
a consequence of the results of Neeman \cite{Neeman-1} 
and Hall-Rydh \cite{HR1}:

\begin{prop}\label{prop:Com-perf}
Let $\sX$ be a nice quotient stack.
Then a chain complex $P \in  Ch_{qc}(\sX)$ is perfect if and only if it is
a compact object of $D_{qc}(\sX)$. 
\end{prop}
\begin{proof}
Suppose $P$ is compact. We need to show that
$s^*(P)$ is perfect on $U = \Spec(A)$ for every smooth map 
$s: U \to \sX$. 
Since the compact objects of $D_{qc}(U)$ are perfect, 
it suffices to show that $s^*(P)$ is compact.
We deduce this using \cite[Theorem~5.1]{Neeman-1}. 

The push-forward functor 
${\bf R}s_*: D_{qc}(U) \to D_{qc}(\sX)$ is a right adjoint to the pull-back 
${\bf L}s^*:  D_{qc}(\sX) \to D_{qc}(U)$. As ${\bf R}s_*$ and ${\bf L}s^*$
both preserve small coproducts (see the proof of 
\lemref{lem:cpt-obj-D_(qc)(X on Z)} below), 
it follows from \cite[Theorem~5.1]{Neeman-1} that 
$s^*(P)$ is compact.

If $P$ is perfect, then it is a compact object of $D_{qc}(\sX)$
by our assumption on $\sX$ and \cite[Theorem~C]{HR1}.
\end{proof}

\begin{lem} \label{lem:cpt-obj-D_(qc)(X on Z)}
Let $\sX$ be a nice quotient stack and let $\sZ \subset \sX$ be a closed 
substack. Then the compact objects of $D_{qc,\sZ}(\sX)$ 
are exactly those which are perfect in $Ch_{qc}(\sX)$.
\end{lem}
\begin{proof}
It follows from \propref{prop:Com-perf} that $D^{\rm perf}_{\sZ}(\sX)
\subseteq D^c_{qc,\sZ}(\sX)$.
To prove the other inclusion, let $K \in D^c_{qc,\sZ}(\sX)$.
We need to show that $K$ is a perfect complex in $D_{qc}(\sX)$.
Let $s: V = \Spec(A) \to \sX$ be any smooth morphism and
set $T = s^{-1} (\sZ)$.
Consider a set of objects $\{F_{\alpha} \}$ in $D_{qc,T}(V)$.
Since $\sX$ is a quotient stack, there exists a smooth atlas
$u: X \to \sX$, where $X \in \Sch_k$.
This gives a 2-Cartesian square of stacks
\begin{equation} \label{eqn:Base-chng}
\xymatrix@C1pc{
W \ar[r]^{s'} \ar[d]_{u'} & X \ar[d]^{u} \\
V \ar[r]_{s} & \sX.}
\end{equation}

The maps $u$ and $s$ are tor-independent because they are smooth.
Since $\Delta_{\sX}$ is representable and $V$ is affine, it follows
that $s$ is representable. We conclude from 
\cite[Lemma~2.5(3), Corollary~4.13]{HR} that 
$u^*{\bf R}s_*(F_{\alpha}) \overset{\simeq} \to {\bf R}s'_* u'^*(F_{\alpha})$.
It follows that ${\bf R}s_*(F_{\alpha}) \in D_{qc, \sZ}(\sX)$. 
Using adjointness \cite[Lemma 3.3]{AK}, we get 
\[
\begin{array}{llll}
\Hom_{D_{qc,T}(V)}(s^*(K), \oplus_{\alpha} F_{\alpha}) & \simeq & 
\Hom_{D_{qc,\sZ}(\sX)} (K, {\bf R}s_* (\oplus_{\alpha}F_{\alpha})) &
 \\
& {\simeq}^1 &  
\Hom_{D_{qc,\sZ}(\sX)} (K, \oplus_{\alpha} {\bf R}s_* (F_{\alpha})) & \\
& {\simeq}^2 & 
\oplus_{\alpha} \Hom_{D_{qc,\sZ}(\sX)}(K,  {\bf R}s_*( F_{\alpha})) & \\
& \simeq & 
\oplus_{\alpha} \Hom_{D_{qc,T}(V)}(s^*(K),  F_{\alpha}), & \\
\end{array}
\] 
where ${\simeq}^1$ follows from the fact that ${\bf R}s_*$ preserves
small coproducts (see \cite[Lemma~2.5(3), Lemma~2.6(3)]{HR}), 
and $\simeq^2$ follows since $K \in D^c_{qc,Z}(\sX)$.
This shows that $s^*(K) \in {D_{qc,T}^c(V)}$.
Since $V$ is affine, this implies that $s^*(K)$ is perfect.
\end{proof}

\subsection{Excision for derived category}\label{sec:Car.Eil-resln}
We now prove an excision property for the derived category
of perfect complexes on stacks using the technique of
Cartan-Eilenberg resolutions.

Let $\sA$ be a Grothendieck category and let $D(\sA)$ denote 
the unbounded derived category of $\sA$.
Let $Ch(\sA)$ denote the category of all chain complexes over $\sA$.
An object $A \in Ch(\sA)$ is said to be {\it $K$-injective}  if for every 
acyclic complex $J \in Ch(\sA)$, the complex $\Hom^{\cdot}(J, A)$ is acyclic. 
Since $\sA$ has enough injectives, 
a complex over $\sA$ has a Cartan-Eilenberg resolution 
(see \cite[$0_{III}$(11.4.2)]{EGA3}). 

It is known that a Cartan-Eilenberg resolution of an unbounded complex over 
$\sA$ need not, in general, be a $K$-injective resolution.
However, when $\sX$ is a scheme or 
a Noetherian and separated Deligne-Mumford stack 
over a fixed Noetherian base scheme, it has been shown that
for a complex $J$ of $\sO_{\sX}$-modules with quasi-coherent cohomology,
the total complex of a Cartan-Eilenberg resolution does give a 
$K$-injective resolution of $J$ (see \cite{Keller}, \cite[Proposition 2.2]{AK}).
Our first objective is to extend these results to all algebraic 
stacks. We follow the techniques of \cite{AK} closely.
Given a double complex $J^{\bullet,\bullet}$, let $\widehat{\Tot(J)}$ 
denote the (product) total complex.

\begin{prop}\label{prop:Cartan-Eilenberg-D_(qc)(X)}
Let $\sX$ be a stack and let $K \in Ch_{qc}(\sX)$. 
Let $E \xrightarrow{\epsilon} I^{\bullet, \bullet}$ be a Cartan-Eilenberg 
resolution of $E$ in $Ch(\sX)$. 
Then $E \xrightarrow{\epsilon} \widehat{\Tot(I)}$ is a 
$K$-injective resolution of $E$.
\end{prop}
\begin{proof}
Since $Mod(\sX)$ is a Grothendieck category and $I^{\bullet, \bullet}$ is a
Cartan-Eilenberg resolution, $\widehat{\Tot(I)}$ is a $K$-injective complex 
by \cite[A.3]{Weibel}. We only need to show that 
$E \xrightarrow{\epsilon} \widehat{\Tot(I)}$ is a quasi-isomorphism.
Let $\tau^{\geq p}(E) := 0 \to E^p/B^pE \to E^{p+1} \to \cdots$
denote the good truncation of $E$.
Then $\{\tau^{\geq p}(E) \}_{p \in \Z}$ gives an inverse system
of bounded below complexes with surjective maps and such that
$E \xrightarrow{\simeq} \varprojlim_p \tau^{\geq p}(E)$.
Let $\tau^{\geq p}(I)$ denote the double complex whose $i$-th row 
is the good truncation of the $i$-th row  of $I^{\bullet, \bullet}$
as above.

Let $L_p^{\bullet,\bullet} = {\rm Ker} (\tau^{\geq p}(I) \surj 
\tau^{\geq {p+1}}(I))$.
Then $ I^{\bullet,\bullet} \surj \tau^{\geq {p}}(I) \surj \tau^{\geq {p+1}}(I)$ 
and $ I^{\bullet,\bullet} \xrightarrow{\simeq} \varprojlim_p \tau^{\geq p}(I)$.
Therefore, 
$ \widehat{\Tot(I)} \xrightarrow{\simeq} 
\varprojlim_p \widehat{\Tot (\tau^{\geq p}(I))}. $
Moreover, since $\tau^{\geq p}(I)$ is a Cartan-Eilenberg resolution of
the bounded below complex $\tau^{\geq p}(E)$, it is known that 
for each $p \in \Z$, $\tau^{\geq p}(E) 
\xrightarrow{\epsilon_p} \widehat{\Tot(\tau^{\geq p}(I))}$
is a quasi-isomorphism.

Also, the standard properties of Cartan-Eilenberg resolutions imply that 
$B^pE \to B^pI^{\bullet, \bullet}$ is an injective resolution and 
hence, the inclusions $B^pI^{\bullet, i} \inj I^{\bullet, i}$ are all split.
In particular, the maps $\tau^{\geq {p}}(I) \surj \tau^{\geq {p+1}}(I)$
are term-wise split surjective.
Since $\tau^{\geq {p}}(I)$ are upper half plane complexes with bounded below 
rows, we conclude that the sequences
\begin{equation}
0 \to \widehat{\Tot(L_p)} \to \widehat{\Tot(\tau^{\geq p}(I))} \to
\widehat{\Tot(\tau^{\geq {p+1}}(I))} \to 0
\end{equation}
are exact and are split in each degree.

Hence, we see that $\widehat{\Tot(I)} \xrightarrow{\simeq} 
\varprojlim_p \widehat{\Tot(\tau^{\geq p}(I))}$, where each 
$\widehat{\Tot(\tau^{\geq p}(I))}$ is a bounded below complex of 
injective $\sO_{\sX}$ -modules, and 
$\epsilon$ is induced by a compatible system  
of quasi-isomorphisms $\epsilon_p$. Furthermore, 
$\widehat{\Tot(\tau^{\geq p}(I))} \to 
\widehat{\Tot(\tau^{\geq {p+1}}(I))}$
is degree-wise split surjective with kernel $\widehat{\Tot(L_p)}$, 
which is a bounded below complex of injective $\sO_{\sX}$-modules. 
Since $\sH_i(E) \in QC(\sX)$ and $QC(\sX) \subseteq Mod(\sX)$ satisfies
Assumption 2.1.2 of \cite{LO}, the proposition now follows from 
\cite[Proposition 2.1.4]{LO}.
\end{proof}

\begin{cor} \label{cor:lim-of-good-trunc}
 Let $f: \sY \to \sX$ be a morphism of stacks
 and let $E \in D_{qc}(\sY)$. Then the natural map
 ${\bf R}f_*(E) \to \varprojlim_{n} {\bf R}f_* (\tau ^{\ge n}(E) )$
 is an isomorphism in $D_{qc}(\sX)$.
 \end{cor}
 \begin{proof}
 This is easily checked by replacing $E$ by a Cartan-Eilenberg resolution
 and using properties of Cartan-Eilenberg resolutions and good truncation.
 \end{proof}

Recall that a morphism $f: \sY \to \sX$ of stacks is {\it representable}, 
if for every algebraic space $T$ and a morphism $T \to \sX$, the fiber product 
$T {\underset{\sX} \times} \sY$ is represented by an algebraic space.
If $T {\underset{\sX} \times} \sY$ is represented by a scheme 
whenever $T$ is a scheme, we say that $f: \sY \to \sX$ is {\it strongly representable}.

\begin{prop} \label{prop:excision-der-cat}
Let $f: \sY \to \sX$ be a strongly representable \'etale morphism of 
stacks. Let $\sZ \overset{i} {\inj} \sX$ be a closed substack
such that $f : \sZ {\underset {\sX} \times} \sY \to \sZ$ induces an 
isomorphism of the associated reduced stacks.
Then $f^* : D_{qc,\sZ} (\sX) \to D_{qc,\sZ {\underset{\sX} \times} \sY} (\sY)$
is an equivalence.
\end{prop}
\begin{proof}
We set $\sW= \sZ {\underset{\sX} \times} \sY$.
Let us first assume that $E \in  D^+_{qc,\sZ} (\sX)$.
We claim that the adjunction map 
$E \to {\bf R}f_* \circ f^* (E)$ is an isomorphism.
The proof of this claim is identical to that of \cite[Proposition 3.4]{KO}
which considers the case of schemes and Deligne-Mumford stacks. 
We take a smooth atlas $s: U \to \sX$ with $U \in \Sch_k$
and note that $U {\underset{\sX}\times} \ \sY \to U$ is an
{\'e}tale morphism in $\Sch_k$ because $f$ is strongly representable.  
As in the proof of \cite[Proposition 3.4]{KO}, an application of 
\cite[Corollary~4.13]{HR} now reduces the problem to the case of schemes.
By similar arguments, if $F \in D^-_{qc, \sW} (\sY)$, the co-adjunction map 
$f^* \circ {\bf R}f_* (F) \to F$ is an isomorphism (see proof of 
\cite[Theorem~3.5]{KO} for details).

To prove the proposition, we need to show that $f^*$ is fully faithful and
essentially surjective on objects. 
To prove the first assertion, let $E \in D_{qc,\sZ} (\sX)$.
Since $f^*$ is exact, it commutes with good truncation. 
Applying this to the isomorphism $E \xrightarrow{\simeq} 
{\underset {n}{\varprojlim}}~ {\tau}^{\ge n}(E)$, we conclude
from \corref{cor:lim-of-good-trunc} and what we showed above for the
bounded below complexes that the adjunction map
$E \to {\bf R}f_* \circ f^* (E)$ is an isomorphism.
If $E' \in  D_{qc,\sZ} (\sX)$ is now another object, then
\[
\begin{array}{lll}
\Hom_{D_{qc,\sZ} (\sX)}(E, E') & \simeq &
\Hom_{D_{qc,\sZ} (\sX)}(E, {\bf R}f_* \circ f^* (E')) \\
& \simeq & \Hom_{D_{qc}(\sX)}(E, {\bf R}f_* \circ f^* (E')) \\
& {\simeq}^1 & \Hom_{D_{qc}(\sY)}(f^*(E), f^* (E')) \\
& \simeq & \Hom_{D_{qc, \sW}(\sY)}(f^*(E), f^* (E')) \\
\end{array}
\]
where ${\simeq}^1$ follows from the adjointness of $(f^*, {\bf R}f_*)$
\cite[Lemma~3.3]{AK}. 

To prove the essential surjectivity of $f^*$, let $F \in D_{qc, \sW}(\sY)$.
If $F \in D^-_{qc,\sW} (\sY)$, then we have shown above that the map
$f^* \circ {\bf R}f_*(F) \to F$ is an isomorphism. The general case 
follows from the bounded above case using the isomorphism
${\underset {n}{\varinjlim}}~ {\tau}^{\le n}(F) \overset{\simeq} \to F$.
\end{proof}

\section{Algebraic $K$-theory of nice quotient 
stacks}\label{sec:KTH}
In this section, we prove \thmref{thm:Main-1}.
Let $\sX$ be a stack. We begin with the definition and
some preliminary results on the $K$-theory spectrum for stacks.

\subsection{$K$-theory spectrum}\label{sec:K-spec}
The algebraic $K$-theory spectrum $K(\sX)$ of $\sX$
is defined to be the $K$-theory spectrum of the
complicial biWaldhausen category of perfect complexes in $Ch_{qc}(\sX)$
in the sense of \cite[\S~1.5.2]{TT}. Here, the complicial biWaldhausen category
structure is given with respect to the degree-wise split monomorphisms as 
cofibrations and quasi-isomorphisms as weak equivalences.
The homotopy groups of the spectrum $K(\sX)$ are defined to be the
$K$-groups of the stack $\sX$ and are denoted by $K_n(\sX)$. 
Note that these groups are $0$ if $n < 0$ (see \cite[\S~1.5.3]{TT}).
We shall extend this definition to negative integers later in this section.
For a closed substack $\sZ$ of $\sX$, 
$K(\sX ~ {\rm on} ~ \sZ)$ is the $K$-theory spectrum of the complicial 
biWaldhausen category of those perfect complexes on $\sX$ which are acyclic
on $\sX \setminus \sZ$.

\begin{lem}\label{lem:K-theory-crisp-stack}
For a stack $\sX$ with affine diagonal, 
the inclusion of the complicial biWaldhausen category of 
perfect complexes of quasi-coherent $\sO_{\sX}$-modules 
into the category of perfect complexes in $Ch_{qc}(\sX)$ 
induces a homotopy equivalence of their $K$-theory spectra.

Similarly for a closed substack $\sZ \inj \sX$, 
$K(\sX ~ {\rm on} ~ \sZ)$ is homotopy equivalent to the $K$-theory spectra
of the complicial biWaldhausen category of 
perfect complexes of quasi-coherent $\sO_{\sX}$-modules which are 
acyclic on $\sX \setminus \sZ$.
\end{lem}
\begin{proof}
For a stack $\sX$ with affine diagonal, 
the inclusion functors $\Phi: Ch(QC(\sX)) \to Ch_{qc}(\sX)$ 
and $\Phi_{\sZ}: Ch_{\sZ}(QC(\sX)) \to Ch_{qc, \sZ}(\sX)$ 
induce equivalences of their left bounded derived categories 
by \cite[Theorem 3.8]{Lur}.
Therefore, they restrict to the equivalences of the derived homotopy categories
of the biWaldhausen categories of perfect complexes of 
quasi-coherent $\sO_{\sX}$-modules (resp. with support in $\sZ$) and that 
of perfect complexes in $Ch_{qc}(\sX)$ (resp. with support in $\sZ$). 
By \cite[Theorem 1.9.8]{TT}, these inclusions therefore
induce homotopy equivalence of their $K$-theory spectra.
\end{proof}

\begin{lem} \label{lem:K-theory-Q.stack}
Let $\sX$ be a quotient stack with resolution property.
Consider the following list of complicial biWaldhausen categories:
\begin{enumerate}
\item bounded complexes of vector bundles on $\sX$.
\item perfect complexes in $Ch(QC(\sX))$.
\item perfect complexes in $Ch_{qc}(\sX)$.
\end{enumerate}
Then the obvious inclusion functors induce homotopy equivalences of all
their $K$-theory spectra. Furthermore, $K(\sX)$ is homotopy
equivalent to the algebraic $K$-theory spectrum of the exact category of vector bundles
on $\sX$.
\end{lem}
\begin{proof}
The inclusion of (1) in (2) induces a
homotopy equivalence of $K$-theory spectra by
\propref{prop:perfect-complex-[X/G]} and \cite[Theorem 1.9.8]{TT}.
The inclusion of (2) in (3) induces homotopy equivalence of $K$-theory spectra
by \lemref{lem:K-theory-crisp-stack}. The last assertion follows from
\cite[Theorem~1.11.7]{TT}.
\end{proof}

\subsection{The localization and excision for 
$K$-theory}\label{subsec:Loc-thm}
We now establish the localization sequence and excision for the $K$-theory
of nice quotient stacks. We begin with the following localization at the level of
$D_{qc}(\sX)$.

\begin{prop}\label{prop:Verdier-quotient}
Let $\sX$ be a nice quotient stack and let
$\sZ \inj X$ be a closed substack with open
complement $j: \sU \inj \sX$. Assume that $\sX$ has the resolution property.
Then the following hold.
\begin{enumerate}
\item
$D_{qc}(\sX)$,  $D_{qc,\sZ}(\sX)$ and $D_{qc}(\sU)$ are compactly generated.
\item
The functor
\[
j^*: \frac{D_{qc}(\sX)}{D_{qc,\sZ}(\sX)} \to D_{qc}(\sU)
\]
is an equivalence of triangulated categories.
\end{enumerate}
\end{prop}
\begin{proof}
The stack $\sU$ has the resolution property by our assumption and
\cite[Theorem~A]{Gross}. 
It follows from \propref{prop:Com-perf} that every perfect complex in
$D_{qc}(\sX)$ is compact, i.e., $\sX$ is {\sl concentrated}. Since
$\sX$ and $\sU$ have affine diagonal with resolution property, it follows from
\cite[Proposition~8.4]{HR} that $D_{qc}(\sX)$, $D_{qc,\sZ}(\sX)$ and
$D_{qc}(\sU)$ are compactly generated.

(2) is an easy consequence of adjointness of the functors
$(j^*, {\bf R}j_*)$ and works exactly like the case of schemes.
One checks easily that $j^*$ is fully faithful and $j^* \circ {\bf R}j_*$
is identity on $D_{qc}(\sU)$. 
\end{proof}

\begin{thm} [\bf{Localization sequence}] 
\label{thm:equiv-localization-thm}
Let $\sX$ be a nice quotient stack and let
$\sZ \inj X$ be a closed substack with open
complement $j: \sU \inj \sX$. Assume that $\sX$ has the resolution property.
Then the morphism of spectra
$K(\sX ~ {\rm on} ~ \sZ) \to K(\sX) \to K(\sU)$
induce a long exact sequence
\begin{align*}
\cdots \to K_i(\sX ~ {\rm on} ~ \sZ) \to K_i(\sX) & \to K_i(\sU) \to 
K_{i-1}(\sX ~ {\rm on} ~ \sZ) \to \cdots \\
& \to K_0(\sX ~ {\rm on} ~ \sZ) \to K_0(\sX) \to K_0(\sU).
\end{align*}
\end{thm}
\begin{proof}
It follows from \propref{prop:Com-perf}, \lemref{lem:cpt-obj-D_(qc)(X on Z)}
and \propref{prop:Verdier-quotient} 
that there is a commutative diagram of triangulated categories
\begin{equation}\label{eqn:Verdier-quotient-2}
\xymatrix@C-=0.5cm{
D_{\sZ}^{\rm perf}(\sX) \ar@{^{(}->}[d] \ar[r]
& D^{\rm perf}(\sX) \ar@{^{(}->}[d]  \ar[r]
& D^{\rm perf}(\sU) \ar@{^{(}->}[d] \\
D_{qc,\sZ}(\sX) \ar[r]
& D_{qc}(\sX) \ar[r]
& D_{qc}(\sU),\\
}
\end{equation}
where the bottom sequence is a localization sequence of
triangulated categories and the top row is the sequence of
full subcategories of compact objects of the corresponding
categories in the bottom row. Moreover, each triangulated category in the
bottom row is generated by its compact objects in the top arrow.
We can thus apply \cite[Theorem 2.1]{Neeman} to conclude that
the functor
\begin{equation}\label{eqn:Verdier-quotient-3}
\frac{D^{\rm perf}(\sX)}{D^{\rm perf}_{\sZ}(\sX)} \to D^{\rm perf}(\sU)
\end{equation}
is fully faithful, and an equivalence up to direct factors.

Let $\Sigma$ be the category whose objects are perfect complexes in 
$Ch_{qc}(\sX)$, and where a map $x \to y$ is a weak equivalence if the
restriction $x|_{\sU} \to y|_{\sU}$ is a quasi-isomorphism in  $Ch_{qc}(\sU)$. 
The cofibrations in $\Sigma$ are degree-wise split monomorphisms.
Then it is easy to see that $\Sigma$ is a complicial biWaldhausen model 
for the quotient category 
$\frac{D^{\rm perf}(\sX)}{D^{\rm perf}_{\sZ}(\sX)}$. 
Thus, by Waldhausen localization Theorem \cite[1.8.2, Theorem 1.9.8]{TT}, 
there is a homotopy fibration of spectra: 
$K(\sX ~ {\rm on} ~ \sZ) \to K(\sX) \to K(\Sigma)$.
It follows from ~\eqref{eqn:Verdier-quotient-3} and 
\cite[Lemma 0.6]{Neeman} that 
$K(\Sigma) \to K(\sU)$ is a covering map of spectra.
In particular, $K_i(\Sigma) \xrightarrow{\simeq} K_i(\sU)$ for $i \geq 1$
and $K_0(\Sigma) \inj K_0(\sU)$. 
\end{proof}

\begin{thm} [\bf{Excision}] \label{thm:equiv-excision}
Let $\sX$ be a nice quotient stack and let
$\sZ \inj \sX$ be a closed substack. 
Let $f: \sY \to \sX$ be a strongly representable \'etale morphism of 
stacks such that $f : \sZ {\underset {\sX} \times} \sY \to \sZ$  
induces an isomorphism of the associated reduced stacks.
Assume that $\sX, \sY$ have the resolution property.
Then $f^*$ induces a homotopy equivalence 
\[
f^*: K(\sX \ {\rm on} \ \sZ) \xrightarrow{\simeq} K(\sY \ {\rm on} \ 
\sZ {\underset{\sX}\times}\ \sY).
\]
\end{thm}
\begin{proof}
We first observe that since $f$ is strongly representable, $\sY$ is also a nice quotient
stack. The theorem now follows directly from \lemref{lem:cpt-obj-D_(qc)(X on Z)}
and \propref{prop:excision-der-cat}.
\end{proof}

\subsection{Projective bundle formula} \label{sec:Proj-bundle-thm}
In order to define the non-connective $K$-theory of stacks, 
we need the projective bundle formula for their $K$-theory.
This formula for the equivariant $K$-theory 
was proven in \cite[Theorem~3.1]{Thom2}.
We adapt the argument of Thomason to extend it
to the $K$-theory of all stacks. Though this formula is used in this 
text only for quotient stacks, its most general form plays a
crucial role in \cite{HK}.
For details on the projective bundles over algebraic 
stacks, see \cite[Chapter 14]{LMB}.

\begin{thm}\label{thm:proj-bdle-thm}
Let $\sX$ be a stack, $\sE$ a vector bundle 
of rank $d$ and $p:\P \sE \to \sX$ the projective
bundle associated to it. Let $\sO_{\P \sE} (1)$ be the fundamental
invertible sheaf on $\P \sE$ and $\sO_{\P \sE} (i)$ its $i$-th power in
the group of invertible sheaves over $\sX$.
 
 Then the morphism of $K$-theory spectra induced by the exact functor
 that sends a sequence of $d$ perfect complexes in $Ch_{qc}(\sX)$,
 $(E_0, \cdots, E_{d-1})$ to the perfect complex
 $$
 p^* E_0 \oplus \sO_{\P \sE} (-1) \otimes p^*E_1 \oplus \cdots \oplus
 \sO_{\P \sE} (1-d) \otimes p^*E_{d-1},
 $$
 induces a homotopy equivalence:
 $$
 \Phi : {\underset{d} \prod} K(\sX) \xrightarrow{\sim} K(\P \sE).
 $$
 Similarly, for each closed substack $\sZ$, the exact functor restricts to the
 sub-category of complexes acyclic on $\sX \setminus \sZ$ to give a homotopy
 equivalence:
 $$
 \Phi : {\underset{d} \prod} K(\sX ~ {\rm on} ~ \sZ) \xrightarrow{\sim} 
 K(\P \sE ~ {\rm on} ~ \P (\sE |\sZ)).
 $$
 \end{thm}

We need the following steps to prove this theorem.

\begin{lem} \label{lem:Proj-bdle-1}
 Under the hypothesis of \thmref{thm:proj-bdle-thm}, let $F$ be a perfect
 complex in $Ch_{qc}(\sX)$ or in general a complex with quasi-coherent and
 bounded cohomology. Then the canonical adjunction morphism (\ref{eqn:P1})
 is a quasi-isomorphism:
 \begin{equation} \label{eqn:P1}
 \eta: F {\overset {\sim} \to} {\bf R}p_* p^* F = 
{\bf R}p_*(\sO_{\P \sE} \otimes p^* F).
 \end{equation}
 
 In addition, for $j = 1,2, \cdots, d-1$, we have as a result of cancellation:
 \begin{equation} \label{eqn:P2}
{\bf R}p_* (\sO_{\P \sE} (-j) \otimes p^*F) \simeq 0.
 \end{equation}
\end{lem}
\begin{proof}
The assertion of the lemma is fppf local on $\sX$.
Let $u: U \to \sX$ be a smooth atlas for $\sX$, where $U$ is a scheme.
Since $p: \P \sE \to \sX$ is strongly representable, we can apply 
\cite[Lemma~2.5 (3), Corollary~4.13]{HR} to reduce to the case when 
$\sX \in \Sch_k$. In this latter case, the lemma is proven in 
\cite[Lemma 3]{Thom2}. 
\end{proof}

\begin{lem}\label{lem:Proj-bdle-2}
Under the hypothesis of \thmref{thm:proj-bdle-thm}, if $E$ is a perfect complex 
in $Ch_{qc}(\P \sE)$, then the following hold.
\begin{enumerate}
\item
${\bf R}p_*(E)$ is a perfect complex in $Ch_{qc}(\sX)$.
\item
If ${\bf R}p_*(E \otimes \sO_{\P \sE}(i))$ is acyclic 
on $\sX$ for $i = 0, 1, \cdots, d-1$, then $E$ is acyclic on $\P \sE$.
\end{enumerate}
\end{lem}
\begin{proof}
Since the assertion is fppf local on $\sX$ and the perfectness is
checked by base change of $\sX$ by smooth morphisms from affine schemes,
we can use \cite[Lemma~2.5 (3), Corollary~4.13]{HR} again to replace $\sX$ by
a scheme. The part (1) then follows from \cite[Lemma 4]{Thom2} and (2)
follows from \cite[Lemma~5]{Thom2}.
\end{proof}

\begin{proof} [Proof of \thmref{thm:proj-bdle-thm}] 
The proof follows exactly along the lines of the proof of
\cite[Theorem~1]{Thom2}, using 
Lemmas~\ref{lem:Proj-bdle-1} and ~\ref{lem:Proj-bdle-2}, which
generalize \cite[Lemmas~3, 4, 5]{Thom2} to stacks.
\end{proof}

\subsection{$K$-theory of regular blow-ups of 
stacks}\label{sec:RBUP}
A closed immersion $\sY \to \sX$ of stacks over $k$ is defined
to be a {\it regular immersion} of codimension $d$ if there exists a smooth
atlas $U \to \sX$ of $\sX$ such that $\sY \times_{\sX} U \to U$ is a regular 
immersion of schemes of codimension $d$. This is well defined as $U$ is
Noetherian and regular immersions behave well under flat base change 
and satisfy fpqc descent.
For a closed immersion $i: \sY \to \sX$, the blow-up of $\sX$ along $\sY$
is defined to be 
$p: \widetilde{\sX} = \Proj (\bigoplus_{n \geq 0} \sI_{\sY}^n) \to \sX$.
See \cite[Chapter 14]{LMB} for relative proj construction on stacks.
Note that in case of a regular immersion, 
$\widetilde{\sX} \times_{\sX} \sY \to \sY$ is a projective bundle over $\sY$,
similar to schemes.

\begin{thm} \label{thm: K-theory-descent-reg-blow-up} 
Let $i: \sY \to \sX$ be a regular immersion of codimension $d$ of 
stacks. Let $p: \sX' \to \sX$ be the blow-up of $\sX$ along $\sY$
and $j: \sY' = \sY \times_{\sX} \sX' \to \sX'$, $q: \sY' \to \sY$
be the maps obtained by base change. Then the square
\begin{equation}\label{eqn:BU0}
\xymatrix@C1pc{
K(\sX) \ar[r]^{i^*} \ar[d]_{p^*} & K(\sY) \ar[d]^{q^*} \\
K(\sX') \ar[r]_{j^*} & K(\sY').}
\end{equation}
is homotopy Cartesian.
\end{thm}
\begin{proof}
This is proved in \cite[Proposition~1.5]{CHSW} in the case of schemes and
an identical proof works for the case of stacks, in the presence of the 
results of \ref{sec:Proj-bundle-thm} and \lemref{lem:Blow-up}. 
We give some details on the strategy of 
the proof. 
For $r = 0, \cdots, d-1$, let $D_r^{\rm perf}(\sX') \subset D^{\rm perf}(\sX')$
be the full triangulated subcategory generated by $Lp^*F$ and 
$Rj_*Lq^*G \otimes \sO_{\sX'}(-l)$ for $F \in D^{\rm perf}(\sX)$, $G \in D^{\rm perf}(\sY)$
and $l = 1, \cdots, r$. 
Let $D_r^{\rm perf}(\sY') \subset D^{\rm perf}(\sY')$
be the full triangulated subcategory generated by
$Lq^*G \otimes \sO_{\sY'}(-l)$ for $G \in D^{\rm perf}(\sY)$ and $l = 0, \cdots, r$.
By Lemmas \ref{lem:Proj-bdle-1} and \ref{lem:Blow-up}(1), 
$Lp^*: D^{\rm perf}(\sX) \to D_0^{\rm perf}(\sX')$ and
$Lq^*: D^{\rm perf}(\sY) \to D_0^{\rm perf}(\sY')$ are equivalences.
Exactly as in \cite[Lemma~1.2]{CHSW},
one shows that $D_{d-1}^{\rm perf}(\sX') = D^{\rm perf}(\sX')$  and 
$D_{d-1}^{\rm perf}(\sY') = D^{\rm perf}(\sY')$
using Lemmas \ref{lem:Proj-bdle-2} and \ref{lem:Blow-up}.

To prove the theorem, it is sufficient to show that 
$Lj^*$ is compatible with the filtrations on $D^{\rm perf}(\sX')$ and $D^{\rm perf}(\sY')$:
\begin{equation}\label{eqn:BU1}
\xymatrix@C-=0.5cm{
D^{\rm perf}(\sX) \ar^{Li^*}[d] \ar^{Lp^*}_{\sim}[r]
& D_0^{\rm perf}(\sX') \ar^{Lj^*}[d]  \ar@{^{(}->}[r]
& D_1^{\rm perf}(\sX') \ar^{Lj^*}[d]  \ar@{^{(}->}[r]
& \cdots \ar@{^{(}->}[r]
&  D_{d-1}^{\rm perf}(\sX') = D^{\rm perf}(\sX') \ar^{Lj^*}[d]\\
D^{\rm perf}(\sY)  \ar^{Lq^*}_{\sim}[r]
& D_0^{\rm perf}(\sY')  \ar@{^{(}->}[r]
& D_1^{\rm perf}(\sY')   \ar@{^{(}->}[r]
& \cdots \ar@{^{(}->}[r]
&  D_{d-1}^{\rm perf}(\sY') = D^{\rm perf}(\sY'), 
}
\end{equation}
and that for $r = 0, \cdots ,d-2$, $Lj^*$ induces equivalences on quotient 
triangulated categories:
\[
Lj^*: D_{r+1}^{\rm perf}(\sX')/ D_r^{\rm perf}(\sX') {\overset {\sim} \to}
D_{r+1}^{\rm perf}(\sY')/ D_r^{\rm perf}(\sY').
\]
Given this, it follows from \cite[Theorems 1.8.2, 1.9.8]{TT} that every square in \eqref{eqn:BU1}
induces a homotopy Cartesian square of $K$-theory spectra.

To prove the compatibility of $Lj^*$, it is enough to check on generators and in
this case, it can be reduced to the case of schemes using \cite[Corollary~4.13]{HR}.
To prove that $Lj^*$ induces equivalence on quotients,
we first note that the composition 
\[
Lj^* \circ [\sO_{\sX'}(-r-1) \otimes Rj_*Lq^*]:  D^{\rm perf}(\sY) \to 
D_{r+1}^{\rm perf}(\sX')/ D_r^{\rm perf}(\sX') \to 
D_{r+1}^{\rm perf}(\sY')/ D_r^{\rm perf}(\sY')
\]
agrees with 
$\sO_{\sY'}(-r-1) \otimes Lq^*: D^{\rm perf}(\sY) \to 
D_{r+1}^{\rm perf}(\sY')/ D_r^{\rm perf}(\sY')$,
up to a natural equivalence.
This follows as in the proof of \cite[Lemma~1.4]{CHSW} using \cite[Corollary~4.13]{HR}.
Therefore, it is enough to show that the functors $\sO_{\sX'}(-r-1) \otimes Rj_*Lq^*$ and
$\sO_{\sY'}(-r-1) \otimes Lq^*$ are equivalences.
But the proof of this 
is exactly the same as the one in \cite[Proposition~1.5]{CHSW} for schemes.
\end{proof}

\begin{lem} \label{lem:Blow-up}
Under the hypotheses of \thmref{thm: K-theory-descent-reg-blow-up},
the following hold.
\begin{enumerate}
\item Let $F$ be a perfect complex on $\sX$. Then the canonical adjunction
morphism \eqref{eqn:BU2} is a quasi-isomorphism:
\begin{equation} \label{eqn:BU2}
\eta: F {\overset {\sim} \to} Rp_*Lp^*F = Rp_*(\sO_{\sX'} \otimes Lp^*F).
\end{equation}
\item Let $r$ be an integer such that $1 \leq r \leq d-1$.
Let $\sA'_r$ denote the full triangulated subcategory of 
$D^{\rm perf}(\sX')$
of those complexes $E$ for which $Rp_*(E \otimes \sO_{\sX'}(i)) \simeq 0$
for $0 \leq i < r$. Then there exists a natural transformation $\partial$ of
functors from $\sA'_r$ to $D^{\rm perf}(\sX')$:
\begin{equation}
\partial: (Rj_*Lq^*Rq_*(E \otimes_{\sO_{\sX'}} \sO_{\sY'}(r-1))) \otimes
\sO_{\sX'}(-r)) [-1] \to E.
\end{equation}
Moreover, $Rp_*(\partial \otimes \sO_{\sX'}(i))$ is a quasi-isomorphism
for $0 \leq i < r+1$.
\item Let $E \in D^{\rm perf}(\sX')$ such that $Rp_*(E \otimes \sO_{\sX'}(i))$
is acyclic on $\sX$ for $i = 0, \cdots, d-1$. Then $E$ is acyclic on $\sX'$.
\end{enumerate}
\end{lem}
\begin{proof}
(1) and (3) are proved in \cite{Thom4} for schemes.
The general case can be deduced from this exactly 
as in Lemmas \ref{lem:Proj-bdle-1} and \ref{lem:Proj-bdle-2}.
For (2), the existence of $\partial$ follows from \cite[Lemma 2.4(a)]{Thom4}
as the construction of $\partial$ given there is natural in $\sX$ for schemes.
To check that $Rp_*(\partial \otimes \sO_{\sX'}(i))$ is a quasi-isomorphism
for $0 \leq i < r+1$, we may again assume that $\sX$ is a scheme 
and this case follows from Lemma 2.4(a) of {\it loc. cit.}
\end{proof}

\subsection{Negative $K$-theory of stacks}\label{sec:Neg-K}
Let $\sU \inj \sX$ be an open immersion of stacks over $k$.
Since $K_0(\sX) \to K_0(\sU)$ is not always surjective in the localization 
theorem, we want to introduce a non-connective spectrum $\K(-)$ with 
$K(-)$ as its $(-1)$-connective cover,
so that $\K(\sX \ {\rm on} \ \sZ) \to \K(\sX) \to \K(\sU)$ is a homotopy
fiber sequence for any closed substack $\sZ$ of $\sX$ with complement $\sU$. 
We define $\K$ only in the absolute case below. The construction of  
$\K(\sX \ {\rm on} \ \sZ)$ follows 
similarly, as shown in \cite{TT}. We shall use
the following version of the Bass fundamental theorem for stacks
to define $\K(\sX)$.
The homotopy groups of $\K(\sX)$ will be denoted by
$\K_i(\sX)$.

\begin{thm} [\bf{Bass fundamental theorem}] \label{thm:Bass-fund}
Let $\sX$ be a nice quotient stack with resolution property
and let $\sX[T] = \sX \times \Spec(k[T])$.
Then the following hold.
\begin{enumerate}
\item For $n \ge 1$, there is an exact sequence 
\begin{align*} 
0 \to K_n(\sX) \xrightarrow{(p_1^*, -p_2^*)}  & K_n(\sX[T]) \oplus 
K_n(\sX[T^{-1}]) \\ &  \xrightarrow{(j_1^*, j_2^*)}  K_n(\sX[T, T^{-1}])
 \xrightarrow{\partial_T} K_{n-1}(\sX) \to 0.
\end{align*}
Here $p_1^*, p_2^*$ are induced by the projections $\sX[T] \to \sX$, etc.
and $j_1^*, j_2^*$ are induced by the open immersions 
$\sX[T^{\pm1}] = \sX[T,T^{-1}] \to \sX[T]$,
etc. The sum of these exact sequences for  $n = 1, 2, \cdots$ 
is an exact sequence of graded $K_*(\sX)$-modules.\
\item For $n \ge 0$, $\partial_T: K_{n+1}(\sX[T^{\pm1}]) \to K_n(\sX)$
is naturally split by a map $h_T$ of $K_*(\sX)$-modules. 
Indeed, the cup product with $T \in K_1(k[T^{\pm1}])$ splits $\partial_T$ up to
a natural automorphism of $K_n(\sX)$.\
\item There is an exact sequence for $n = 0$:
\[
0 \to K_n(\sX) \xrightarrow{(p_1^*, -p_2^*)}  K_n(\sX[T]) \oplus 
K_n(\sX[T^{-1}])  \xrightarrow{(j_1^*, j_2^*)} K_n(\sX[T^{\pm1}]).
\]
\end{enumerate}
\end{thm}
\begin{proof}
It follows from \cite[Lemma~2.6]{Thom1} that
$\P^1_{\sX}$ and $\sX[T]$ are nice quotient stacks with resolution property.
It follows from \thmref{thm:proj-bdle-thm} that there is an isomorphism 
$K_*(\P^1_{\sX}) \simeq K_*(\sX) \oplus K_*(\sX)$, 
where the two summands are $K_*(\sX) [\sO]$
and $K_*(\sX) [\sO(-1)]$ with respect to the external product 
$K(\sX) \wedge K(\P^1_{k}) \to K(\P^1_{\sX})$ and with
$[\sO], [\sO(-1)] \in K_0(\P^1_k)$.
As for schemes  \cite[Theorem 6.1]{TT}, Part (1) now follows directly from 
Theorems~\ref{thm:equiv-localization-thm}
and ~\ref{thm:equiv-excision}.

For (2), it suffices to show that the composite map 
$\partial_T(T \cup p^*(~)): K_n(\sX) \to K_{n+1}(\sX[T^{\pm 1}]) \to K_n(\sX)$
is an automorphism of $K_n(\sX)$ for $n \ge 0$. 
By naturality and the fact that $\partial_T$ is a map of $K_*(\sX)$ modules, 
this reduces to showing that
$\partial_T: K_1(k[T^{\pm 1}]) \to K_0(k)$ sends $T$ to $\pm 1$.
But this is well known and (3) follows from (2) using the analogue of 
\cite[Diagram~(6.1.5)]{TT} for stacks.
\end{proof}

\begin{thm}\label{thm:K-G-main}
For a nice quotient stack $\sX$ with resolution property, there is a
spectrum $\K(\sX)$ together with a natural map of spectra
$K(\sX) \to \K(\sX)$ which induces isomorphism 
$K_i(\sX) \xrightarrow{\simeq} 
\K_i(\sX)$ for $i \ge 0$. 

Let $\sY$ be a nice quotient stack with resolution property
and let $f: \sY \to \sX$ be a strongly representable {\'e}tale map.
Let $\sZ \inj \sX$ be a closed substack such that
$\sZ {\underset{\sX}\times} \ \sY \to \sZ$ induces an isomorphism
of the associated reduced stacks. Let $\pi: \P(\sE) \to \sX$ be the projective 
bundle associated to a vector
bundle $\sE$ on $\sX$ of rank $r$. Then the following hold.
\begin{enumerate}
\item
There is a homotopy fiber sequence of spectra
\[
\K(\sX \ {\rm on} \ \sZ) \to  \K(\sX) \to \K(\sX \setminus \sZ).
\]
\item
The map $f^*: \K(\sX \ {\rm on} \ \sZ)  \to 
\K(\sY \ {\rm on} \ \sZ {\underset{\sX}\times} \ \sY)$ is a homotopy 
equivalence.
\item
The map $\prod\limits_{0}^{r-1} \K(\sX) \to  \K(\P(\sE))$
that sends $(a_0, \cdots , a_{r-1})$ to $\underset{i} \Sigma \
\sO[-i] \otimes \pi^*(a_i)$, is a homotopy equivalence.
\end{enumerate}
\end{thm}
\begin{proof}
The construction of the spectrum $\K(\sX)$ is a direct
consequence of \thmref{thm:Bass-fund} 
by the formalism given in $(6.2) - (6.4)$ of \cite{TT}.
Like for schemes, the proof of (1), (2) and (3) is a standard deduction from
Theorems~\ref{thm:equiv-localization-thm}, ~\ref{thm:equiv-excision}
and ~\ref{thm:proj-bdle-thm}, using the inductive construction of
$\K(\sX)$.
\end{proof}

\subsection{Schlichting's negative K-theory} \label{sec:Schlichting}
Negative $K$-theory of complicial biWaldhausen categories 
was defined by Schlichting \cite{Schl}.
Let $\sX$ be a nice quotient stack.
Schlichting's negative $K$-theory spectrum $K^{\rm Scl}(\sX)$
is the $K$-theory spectrum of the Frobenius pair associated to 
the category $Ch_{qc}(\sX)$. 
It follows from \cite[Theorem 8]{Schl} that $K^{\rm Scl}_i(\sX) = K_i(\sX)$ 
for $i \geq 0$. The following result shows that $K^{\rm Scl}_i(\sX)$
agrees with $\K_i(\sX)$ for $i < 0$.

\begin{thm} \label{thm:Bass-neg-k-theory=Schl}
Let $\sX$ be a nice quotient stack with resolution property.
Then there are natural isomorphisms between 
$K^{\rm Scl}_i(\sX)$ and $\K_i(\sX)$ for $i \le 0$.
\end{thm}
\begin{proof}
Let $p : \P^1_{\sX} \to \sX$ be the projection map. Then, we can prove as in 
\thmref{thm:proj-bdle-thm} that the functors 
$p^* : D^{\rm perf}(\sX) \to D^{\rm perf}(\P^1_{\sX})$ and 
$\sO(-1) \otimes p^* : D^{\rm perf}(\sX) \to D^{\rm perf}(\P^1_{\sX})$,
which are induced by maps of their Frobenius models, induce isomorphisms
$
(p^*, \sO(-1) \otimes p^*): K^{\rm Scl}_i(\sX) \oplus K^{\rm Scl}_i(\sX) 
\xrightarrow{\simeq} K^{\rm Scl}_i(\P^1_{\sX})$ for
$i \leq 0$.
It follows from the proof of Bass' fundamental theorem in 
\cite[6.6 (b)]{TT}
that there is an exact sequence of abelian groups
$$
0 \to K^{\rm Scl}_i(\sX) \to K^{\rm Scl}_i(\sX[T]) \oplus 
K^{\rm Scl}_i(\sX[T^{-1}])
\to K^{\rm Scl}_i(\sX[T, T^{-1}]) \to K^{\rm Scl}_{i-1}(\sX) \to 0
$$
for $i \leq 0$. Since $K^{\rm Scl}_0(\sY) = \K_0(\sY)$ for any stack $\sY$, 
the negative $K$-groups coincide.
\end{proof}

\section{Nisnevich Descent for K-theory of quotient 
stacks}\label{sec:Nis-desc}
In this section, we prove Nisnevich descent in a 2-category of stacks
whose objects
are all quotients of schemes by action of a fixed group scheme. 
So let $G$ be a group scheme over $k$.
Let $\Sch^G_k$ denote the category of separated schemes of finite type over $k$
with $G$-action.
The equivariant Nisnevich topology on $\Sch^G_k$ and the homotopy
theory of simplicial sheaves in this topology was defined and
studied in detail in \cite{HKO}.  
As an application of Theorem~\ref{thm:K-G-main}, we shall
show in this section that the $K$-theory of quotient stacks for $G$-actions
satisfies descent in the equivariant Nisnevich topology on $\Sch^G_k$.

\begin{defn}$($\cite[Definition~2.1]{HKO}$)$ \label{def: Nis-cd}
A {\sl distinguished equivariant Nisnevich square} is a 
Cartesian square 
\begin{equation} \label{eqn:N-1}
\xymatrix{
B \ar[d] \ar[r]
& Y \ar[d]^{p} \\
A \ar@{^{(}->}[r]^{j}
& X  \\
}
\end{equation}
in $\Sch^G_k$ such that:
\begin{enumerate}
\item
$j$ is an open immersion, 
\item
$p$  is \'etale, and 
\item
the induced map $(Y \setminus B)_{\red} \to (X \setminus A)_{\red}$ 
of schemes (without reference to the $G$-action) is an isomorphism.
\end{enumerate}
\end{defn}

\begin{remk}\label{remk:G-inv-red}
We remark here that given a Cartesian square of the type 
~\eqref{eqn:N-1} in $\Sch^G_k$, the closed subscheme $(X \setminus A)_{\red}$ 
(or $(Y \setminus B)_{\red}$) may not in general be $G$-invariant, unless
$G$ is smooth. However, it follows from \cite[Lemma~2.5]{Thom} that
we can always find a $G$-invariant closed subscheme $Z \subset X$ such that
$Z_{\rm red} = X \setminus A$. This closed subscheme can be assumed to be
reduced if $G$ is smooth. Using the elementary fact that
a morphism of schemes is {\'e}tale if and only if the
induced map of the associated reduced schemes is {\'e}tale, 
it follows immediately that the condition (3) in Definition~\ref{def: Nis-cd}
is equivalent to:

\vskip.1cm

\hspace*{.3cm} (3') there is a $G$-invariant closed subscheme $Z \subset X$
with support $X \setminus A$ such that the map $Z \times_X Y \to Z$
in $\Sch^G_k$ is an isomorphism. 
\end{remk}

The collection of distinguished equivariant Nisnevich squares 
forms a $cd$-structure in the sense of \cite{Voev}.
The associated Grothendieck topology is called the equivariant 
{\sl Nisnevich} topology. It is also called the $eN$-topology.
It follows from \cite[Theorem 2.3]{HKO} that the equivariant Nisnevich 
$cd$-structure on $\Sch^G_k$ is complete, regular, and bounded. 
We refer to \cite[\S~2]{Voev} for the definition of a complete, 
regular, and bounded $cd$-structure.

Let $\Sch^G_{k/\Nis}$ denote the category of 
$G$-schemes $X$, such that $X$ admits a family of $G$-equivariant ample line bundles,
equipped with the equivariant Nisnevich topology.
Note that all objects of $\Sch^G_{k/\Nis}$ have the resolution property by
\lemref{lem:Thom-Res}.
It follows from \cite[Corollary 2.11]{HKO} that for a sheaf $\sF$ of abelian
groups on $\Sch^G_{k/\Nis}$, $H^i_{G/\Nis}(X, \sF) = 0$ for $i > \dim(X)$.

\begin{defn} \label{def:Splitting}
An equivariant morphism $Y \to X$ in $\Sch^G_k$ \sl{splits} 
if there is a filtration of $X$ by $G$-invariant closed subschemes
\begin{equation}\label{eqn:split*-0}
\emptyset = X_{n+1} \subsetneq X_{n} \subsetneq \cdots \subsetneq X_0 = X, 
\end{equation}
such that for each $j$, the map
$\left(X_j \setminus X_{j+1}\right) \times_X Y \to X_j \setminus X_{j+1}$
has a $G$-equivariant section. If $f$ is \'etale and surjective, 
the morphism is called an equivariant \sl{split \'etale cover} of $X$.
\end{defn}

\subsection{Equivariant Nisnevich covers}\label{sec:ENC}
In \cite[Proposition 2.15]{HKO}, it is shown that
an equivariant \'etale morphism $Y \to X$ in $\Sch^G_k$ is an 
equivariant Nisnevich cover if and only if it splits.
Further, when $G$ is a finite constant group scheme, it is shown that
an equivariant \'etale map $f : Y \to X$ in $\Sch^G_k$ is an equivariant 
Nisnevich cover if and only if for any point $x \in X$, there is a point 
$y \in Y$ 
such that $f(y) = x$ and $f$ induces isomorphisms $k(x) \simeq k(y)$ and 
$S_y \simeq S_x$. Here, for a point $x \in X$, the set-theoretic
stabilizer $S_x \subseteq G$ is defined by $S_x = \{g \in G~|~ g.x = x \}$
\cite[Proposition 2.17]{HKO}. 

Let $G^0$ denote the connected component of the identity element in $G$.
Suppose that $G$ is of the form 
$G = \stackrel{r}{\underset{i =0}\coprod} g_iG^0$,
where $\{e = g_0, g_1, \cdots , g_r\}$ are points in $G(k)$ which
represent the left cosets of $G^0$. In the next proposition, 
we give an explicit description of the equivariant Nisnevich covers
of reduced schemes $X \in \Sch^G_k$.
For $x \in X$, let $\widetilde{S_x} := 
\{g_i~|~ 0\leq i \leq r ~ ; ~ g_i.x = x\}$.

\begin{prop}\label{prop:Nisne-split}
Let $G$ be a smooth affine group scheme over $k$ as above. 
A morphism $f: Y \to X$ in $\Sch^G_k$ 
is an equivariant split \'etale cover of a reduced scheme $X$ if and only if 
for any point $x \in X$, there is a point $y \in Y$ 
such that $f(y) = x$ and $f$ induces isomorphisms $k(x) \simeq k(y)$ and 
$\widetilde{S_y} \simeq \widetilde{S_x}$.
\end{prop}
\begin{proof}
It is clear that a split \'etale $G$-equivariant family of morphisms satisfies
the given conditions. The heart of the proof is to show the converse.

Suppose $Y \xrightarrow{f} X$ is such that
for any point $x \in X$, there is a point $y \in Y$ 
such that $f(y) = x$ and $f$ induces isomorphisms $k(x) \simeq k(y)$ and 
$\widetilde{S_y} \simeq \widetilde{S_x}$.
Let $W$ be the regular locus of $X$. Then $W$ is a $G$-invariant dense 
open subscheme of $X$.
Set $U = Y \times_X W$. 
Notice that $W$ is a disjoint union of its irreducible components and each
$f_U$ being \'etale, it follows that $U$ is regular and hence a disjoint union
of its irreducible components. 

Let $x \in W$ be a generic point of $W$. Then the closure $W_x = 
\ov{\{x\}}$ in $W$ is an irreducible component of $W$. By our assumption,
there is a point $y \in U$ such that 
\begin{equation}\label{eqn:split-0}
f(y) = x, \ k_x \xrightarrow{\simeq} k_y,  \
{\rm and} \ \widetilde{S_y} \xrightarrow{\simeq} \widetilde{S_x}.
\end{equation}
Then the closure $U_y = \ov{\{y\}}$ in $U$ is an irreducible component of 
$U$. Since $U_y \to W_x$ is \'etale and generically an isomorphism, it
must be an open immersion. Thus $f$ maps $U_y$ isomorphically onto an
open subset of $W_x$. We replace $W_x$ by this open subset $f(U_y)$
and call it our new $W_x$. 

Let $GU_y$ be the image of the action morphism $\mu: G \times U_y \to U$.
Notice that $\mu$ is a smooth map and hence open.
This in particular implies that $GU_y$ is a $G$-invariant open subscheme
of $U$ as $U_y$ is one of the disjoint irreducible components of $U$
and hence open. By the same reason, $GW_x$ is a $G$-invariant open
subscheme of $W$.  

Since the identity component $G^0$ is connected, it keeps $U_y$ invariant. 
Therefore, $y \in U$ is fixed by $G^0$ and hence $G$ acts on this point
via its quotient $\ov{G} = G/{G^0}$.
Since each $g_jG^0$ takes $U_y$ onto an irreducible component of $U$
and since $U$ has only finitely many irreducible components which are all
disjoint, we see that $GU_y = U_{0} \amalg U_{1} \amalg \cdots 
\amalg U_{n}$ is a disjoint union of some irreducible components of
$U$ with $U_{0} = U_y$. In particular, for each $U_{j}$, we have
$U_{j} = g_{j_i}G^0U_y = g_{j_i}U_y$ for some $j_i$. 

Since $f$ maps $U_y$ isomorphically onto $W_x$, we conclude from the above 
that $f$ maps each $U_{j}$ isomorphically onto one and only one
$W_{j}$ such that $GW_x = f \left(GU_y\right) =
W_{0} \amalg W_{1} \amalg \cdots \amalg W_{m}$ (with $m \le n$) is a 
disjoint union of open subsets of some irreducible components of $W$ with 
$W_{0} = W_x$. 
The morphism $f$ will map the open subscheme $GU_y$ isomorphically onto the 
open subscheme $GW_x$ if and only if no two components of $GU_y$ are mapped
onto one component of $GW_x$. This is ensured by using the
condition ~\eqref{eqn:split-0}.

If two distinct components of $GU_y$ are mapped onto 
one component of $GW_x$,
we can (using the equivariance of $f$) apply automorphisms by $g_j$'s 
and assume that one of these
components is $U_y$. In particular, we can find
$j \ge 1$ such that
\begin{equation}\label{eqn:split-1}
W_x = f\left(U_y\right) = f\left(U_{j}\right)
= f\left(g_{j_i} U_{y}\right) = g_{j_i} f\left(U_{y}\right) = g_{j_i}W_x.
\end{equation}

But this implies that $g_{j_i} \in \widetilde{S_x}$ and $g_{j_i} \notin 
\widetilde{S_y}$. 
This violates the condition in ~\eqref{eqn:split-0} that 
$\widetilde{S_y}$ and $\widetilde{S_x}$ are isomorphic.     
We have thus shown that the morphism $f$ has a $G$-equivariant splitting 
over a non-empty $G$-invariant open subset $GW_x$.
Letting $X_1$ be the complement of this open subset in $X$ with reduced 
scheme structure, we see that
$X_1$ is a proper $G$-invariant closed subscheme of $X$ and by
restricting our cover to $X_1$, we get a cover for
$X_1$ satisfying the given conditions. 
The proof of the proposition is now completed by the Noetherian induction. 
\end{proof}

\subsection{Equivariant Nisnevich descent}\label{sec:Des}
It is shown in \cite[\S~3]{HKO} 
that the category of presheaves of $S^1$-spectra
on $\Sch^G_{k/\Nis}$ (denoted by $Pres(\Sch^G_{k/\Nis})$) 
is equipped with the global and local injective
model structures. A morphism $f:\sE \to \sE'$ of presheaves of spectra is 
called a {\sl global weak equivalence} if $\sE(X) \to \sE'(X)$ is a weak 
equivalence of $S^1$-spectra for every object $X \in \Sch^G_{k/\Nis}$. 
It is a global injective cofibration if  $\sE(X) \to \sE'(X)$ is a 
cofibration of $S^1$-spectra for every object $X \in \Sch^G_{k/\Nis}$.
The map $f$ is called a {\sl local weak equivalence} if it induces an 
isomorphism on the sheaves of stable homotopy groups of the presheaves of 
spectra in the $eN$-topology. A local (injective)
cofibration is the same as a  global injective cofibration.

A presheaf of spectra $\sE$ on $\Sch^G_{k/\Nis}$ is said to satisfy the 
{\sl equivariant Nisnevich descent} ($eN$-descent) if the fibrant replacement
map $\sE \to \sE'$ in the local injective model structure
of $Pres(\Sch^G_{k/\Nis})$ is a global weak equivalence. 
Let $\sK^G$ denote the presheaf of spectra on $\Sch^G_k$ which 
associates the spectrum $\K([X/G])$ to any $X \in \Sch^G_k$ .
As a consequence of \thmref{thm:K-G-main}, we obtain the following.

\begin{thm}\label{thm:MNDS}
Let $G$ be a nice group scheme over $k$.
Then the  presheaf of spectra $\sK^G$ on $\Sch^G_{k/\Nis}$ 
satisfies the equivariant Nisnevich descent.
\end{thm}
\begin{proof}
Since the $eN$-topology is regular, complete and
bounded by \cite[Theorem~2.3]{HKO}, it suffices to show using
\cite[Proposition~3.8]{Voev} that $\sK^G$ takes a square
of the type ~\eqref{eqn:N-1} to a homotopy Cartesian square of spectra.
In other words, we need to show that the square
\begin{equation}\label{eqn:MV1}
\xymatrix@C1pc{
\K([X/G]) \ar[r]^{j^*} \ar[d]_{p^*} & \K([A/G]) \ar[d]^{p'^*} \\
\K([Y/G]) \ar[r]_{j'^*} & \K([B/G])}
\end{equation}
is homotopy Cartesian.
But this is an immediate consequence of \thmref{thm:K-G-main}.
\end{proof}

\begin{cor}\label{cor:MNDS*}
Let $G$ be a nice group scheme over $k$ 
and let $X \in \Sch^G_{k/\Nis}$.
Then there is a strongly convergent spectral sequence
\[
E^{p,q}_2 = H^p_{eN}\left(X, \sK^G_q\right) \Rightarrow
\K_{q-p}([X/G]).
\]
\end{cor}
\begin{proof} This is immediate from Theorem~\ref{thm:MNDS}
and \cite[Theorem~2.3, Corollary~2.11]{HKO}.
\end{proof}

\section{Homotopy invariance of $K$-theory with 
coefficients for quotient stacks}\label{sec:Finite}
It is known that with finite coefficients,
the ordinary algebraic $K$-theory of schemes 
satisfies the homotopy invariance property 
(see \cite[Theorem~1.2, Proposition~1.6]{Weibel-1} for affine
schemes and \cite[Theorem~9.5]{TT} for the general case).
This is a hard result which was achieved by first defining a homotopy invariant version
of algebraic $K$-theory \cite{Weibel-1} and then showing that
with finite coefficients, this homotopy (invariant) $K$-theory coincides with 
the algebraic $K$-theory.

However, the proof of the agreement between algebraic $K$-theory and
homotopy $K$-theory with finite coefficients requires the knowledge of
a spectral sequence relating $NK$-theory and homotopy $K$-theory
\cite[Remark~1.3.1]{Weibel-1}.
Recall here that $NK(X)$ denotes the homotopy fiber of the pull-back map 
$\iota^*$, where $\iota: X \inj \A^1_k \times X$ denotes the 
$0$-section embedding into the trivial line bundle over a scheme $X$.
The existence of
homotopy $K$-theory for quotient stacks  is not yet known
and one does not know if the above spectral sequence would exist 
for the homotopy $K$-theory of quotient stacks.
In this section, we adopt a different strategy to extend the 
results of Weibel and Thomason-Trobaugh to the $K$-theory of nice
quotient stacks (see \thmref{thm:Fin-coeff}).

\subsection{Homotopy $K$-theory of stacks}\label{sec:HKT}
For $n \in \N$, let $\Delta_n = \Spec\left(
\frac{k[t_0, \cdots , t_n]}{(\sum_i t_i - 1)}\right)$.
Recall that $\Delta_{\bullet} = \{\Delta_n\}_{n \ge 0}$
forms a simplicial scheme whose face and degeneracy maps are
given by the formulas
\[
\partial_r(t_j) = \left\{ \begin{array}{ll}
t_j  & \mbox{if $j < r$} \\
0  & \mbox{if $j = r$} \\
t_{j-1} & \mbox{if $j > r$}
\end{array} \right. \quad
\delta_r(t_j) = \left\{ \begin{array}{ll}
t_j  & \mbox{if $j < r$} \\
t_j + t_{j+1}  & \mbox{if $j = r$} \\
t_{j+1} & \mbox{if $j > r$.}
\end{array} \right.
\]

\begin{defn}\label{defn:HIKT}
For a nice quotient stack $\sX$ with resolution property,
the {\sl homotopy $K$-theory} is defined to be the spectrum
\[
KH(\sX) = {\rm hocolim}_n \K(\sX \times \Delta_n).
\]
\end{defn}

It is clear from the definition that $KH(\sX)$ is contravariant with respect
to morphisms of stacks. Furthermore, there is a natural map of spectra
$\K(\sX) \to KH(\sX)$. 
It is well known that $\K(\sX)$ is not a homotopy invariant functor.
Our first result on $KH(\sX)$ is the following.

\begin{thm}\label{thm:HK-Inv}
Let $\sX$ be a nice quotient stack with resolution property
and let $f: \sE \to \sX$ be a vector bundle morphism.
Then the pull-back map $f^*: KH(\sX) \to KH(\sE)$ is a homotopy 
equivalence.
\end{thm}
\begin{proof}
We first show that the
map $KH(\sX) \to KH(\sX\times \Delta_n)$ is a homotopy equivalence
for every $n \ge 0$. But this is essentially a direct consequence
of the definition of $KH$-theory. By identifying $\Delta_n$ with
$\A^n_k$ and using induction, one needs to show that the map
$KH(\sX) \to KH(\sX[T])$ is a homotopy equivalence.
Proof of this is identical to the case of the 
$KH$-theory of schemes \cite[Theorem~1.2]{Weibel-1}.

To prove the general case, we write $\sX = [X/G]$, where $G$ is a group scheme 
over $k$ acting on a $k$-scheme $X$. We let $E = u^*(\sE)$, where
$u: X \to \sX$ is the quotient map. Then $E$ is a $G$-equivariant vector
bundle on $X$ such that $[E/G] \simeq \sE$. 

We consider the standard fiberwise
contraction map $H: \A^1_k \times E \to E$. Explicitly,
for an open affine $U = \Spec(A) \subseteq X$ over which $f$ is trivial
(without $G$-action), $H|_U$ is induced by the $k$-algebra homomorphism
$A[X_1, \cdots , X_n] \to A[X_1, \cdots , X_n, T]$, given by
$X_j \mapsto TX_j$. It is clear from this that this defines a unique
map $H$ as above which is $G$-equivariant for the trivial $G$-action on 
$\A^1_k$. We have the commutative diagram
\begin{equation}\label{eqn:HK-Inv-0}
\xymatrix@C2pc{
& \{1\} \times E \ar[dl]_{{\rm id}} \ar[d]_{i_1} \ar[dr]^{h_1} & \\
E & \A^1_k \times E \ar[l]_{p} \ar[r]^>>>>>{H} & E \\
& \{0\} \times E, \ar[ul]^{{\rm id}} \ar[u]_{i_0} \ar[ur]_{h_0} &}
\end{equation}
where $h_j = H \circ i_j$ for $j = 0,1$ and $p$ is the projection map.

Let $\iota: X \inj E$ denote the 0-section embedding so that
$f \circ \iota = {\rm id}_X$. So we only need to show that
$f^* \circ \iota^*$ is identity on $KH([E/G])$. Since $h_0 = \iota \circ f$,
it suffices to show that $h^*_0$ is identity.

It follows from the weaker version of homotopy invariance shown
above (applied to $E$) that $p^*$ is an isomorphism on the $KH$-theory of the
stack quotients.
In particular, $i^*_0 = (p^*)^{-1} = i^*_1$.
Since $h_1 = {\rm id}_E$, we get $i^*_1 \circ H^* = {\rm id}$
which in turn yields $H^* = (i^*_1)^{-1} = p^*$ and hence
$h^*_0 = i^*_0 \circ H^* = i^*_0 \circ p^* = {\rm id}$.
This finishes the proof.
\end{proof}

\subsection{Proof of \thmref{thm:Main-2}}\label{sec:Main2*}
The proof of \thmref{thm:Main-2} is a direct consequence of the definition of
$KH(\sX)$ and similar results for the $\K$-theory.
Part (1) of the theorem is \thmref{thm:HK-Inv}.
Part (2) follows directly from Theorems~\ref{thm:K-G-main} and
~\ref{thm: K-theory-descent-reg-blow-up} because the homotopy colimit
preserves homotopy fiber sequences. 

We now prove (3). Let $G$ be a finite group acting on a 
scheme $X$ such that $X$ admits an ample family of line bundles. 
Then $X$ is covered by $G$-invariant affine open subschemes.
By \thmref{thm:MNDS}, it suffices to prove the theorem when $X = \Spec(A)$
is affine. In this case, $\K([X/G])$ is homotopy equivalent to
the $K$-theory of the exact category $\sP^G(A)$ of finitely generated
$G$-equivariant projective $A$-modules (see \lemref{lem:K-theory-Q.stack}).  

Since $G$ is also assumed to be nice, it follows from \cite[Lemma~1.3]{LS} that
$\sP^G(A)$ is equivalent to the exact category $\sP(A^{\rm tw}[G])$ of
finitely generated projective $A^{\rm tw}[G]$-modules.
Recall here that $A^{\rm tw}[G] = \bigoplus_{g \in G} \ Ae_g$
and the product is defined by $(r_g \cdot e_g) (r_h \cdot e_h)
= r_g \cdot (r_h \star g^{-1}) e_{gh}$, where $\star$ indicates the
$G$-action on $A$. 

If $I$ is a nilpotent ideal of $A$ with quotient $B = A/I$, 
it follows from \lemref{lem:twist} that
the map $A^{\rm tw}[G] \to B^{\rm tw}[G]$ is surjective and its
kernel is a nilpotent ideal of $A^{\rm tw}[G]$.
We now apply \cite[Theorem~2.3]{Weibel-1} to conclude that
the map $KH(A^{\rm tw}[G]) \to KH({(A/I)}^{\rm tw}[G])$
is a homotopy equivalence. Since $G$ acts trivially on $\Delta_{\bullet}$, 
there is a canonical isomorphism $(A[\Delta_{\bullet}])^{\rm tw}[G] \simeq
(A^{\rm tw}[G])[\Delta_{\bullet}]$.
We conclude that the map $KH([{\Spec(A)}/G]) \to KH([{\Spec(B)}/G])$ is a 
homotopy equivalence. This finishes the proof.
$\hfill \square$

\begin{lem}\label{lem:twist}
Let $G$ be a finite group acting on commutative unital rings $A$
and $B$. Let $A \surj B$ be a $G$-equivariant surjective ring 
homomorphism whose kernel is nilpotent. Then the induced
map $A^{\rm tw}[G] \to B^{\rm tw}[G]$ is surjective and its
kernel is nilpotent.
\end{lem}

\begin{proof}
Let $I$ denote the kernel of $f: A \surj B$. By hypothesis, there exists an 
integer $n$ such that $I^n = 0$. Since the induced map 
$A^{\rm tw}[G] \to B^{\rm tw}[G]$ is a 
$G$-graded homomorphism induced by $f$ on each graded piece, it is 
a surjection and its kernel is given by 
$I^{\rm tw}[G] = \bigoplus_{g \in G} Ie_g$.
Since $I$ is a $G$-invariant ideal of $A$,
each element of $(I^{\rm tw}[G])^n$ is of the form 
$(a_1.e_{g_1}+ \cdots +a_m.e_{g_m})$, where $g_i \in G$ and $a_i \in I^n$.
Therefore, $(I^{\rm tw}[G])^n = 0$.
\end{proof}

\subsection{$\K$-theory of stacks with coefficients}\label{sec:KG-coeff}
For an integer $n \in \N$, let
\[
\K(\sX ; \Z[1/n]):= {\rm hocolim}(\K(\sX) \xrightarrow{\cdot n}
\K(\sX) \xrightarrow{\cdot n} \cdots) \ \ {\rm and} \ \
\K(\sX ; {\Z}/n):= \K(\sX) \wedge {\mathbb{S}}/{n},
\] 
where ${{\mathbb{S}}}/{n}$ is the mod-$n$ Moore spectrum.
Our final result is the homotopy invariance property of
$\K$-theory with coefficients.

The proof of Theorem \ref{thm:Fin-coeff} uses the
notion of $K$-theory of dg-categories. We briefly
recall its definition and refer to \cite[\S~5.2]{Keller-1} for further
details.
Let $\sA$ be a small dg-category. 
The category $D(\sA)$ is the localization of the category of dg $\sA$-modules 
with respect to quasi-isomorphisms.
The category of perfect objects 
${\rm Per}(\sA)$ is the smallest triangulated subcategory of 
$D(\sA)$ containing the representable objects and closed 
under shifts, extensions and direct factors.
The algebraic $K$-theory of $\sA$ is defined to be the
$K$-theory spectrum of the Waldhausen category ${\rm Per}(\sA)$ where
the cofibrations are the degree-wise split monomorphisms and the
weak equivalences are the quasi-isomorphisms. 

\begin{thm}
\label{thm:Fin-coeff}
Let $\sX$ be a nice quotient stack over $k$ with resolution property
and let $f: \sE \to \sX$ be a vector bundle. Then the following hold.
\begin{enumerate}
\item
For any integer $n$ invertible in $k$, the map
$f^*: \K(\sX ; {\Z}/n) \to \K(\sE ; {\Z}/n)$ is a homotopy equivalence.
\item
For any integer $n$ nilpotent in $k$, the map
$f^*: \K(\sX ; {\Z}[{1}/n]) \to \K(\sE ; {\Z}[1/n])$ is a homotopy equivalence.
\end{enumerate}
\end{thm}
\begin{proof}
The category $\rm Perf(\sX)$ has a natural dg enhancement \cite[Example~5.5]{CT}
whose algebraic $K$-theory (in the sense of $K$-theory of dg-categories)
coincides with $\K(\sX)$ by \cite[Theorem~5.1]{Keller-1}.
It follows from \propref{prop:Com-perf} and \cite[Proposition~8.4]{HR}
that $D_{qc}(\sX)$ is compactly generated and every perfect complex
on $\sX$ is compact. We conclude 
from \cite[Theorem~1.2]{Tabuada} that the theorem holds when
$f$ is the projection map $\sX[T] \to \sX$. To prove the general case, we 
use ~\eqref{eqn:HK-Inv-0} and repeat the 
argument of \thmref{thm:HK-Inv} in verbatim.
\end{proof}

 
\begin{cor}\label{cor:KG-KH}
Let $\sX$ be as in \thmref{thm:Fin-coeff}. Then the following hold.
\begin{enumerate}
\item
For any integer $n$ invertible in $k$, the map
$f^*: \K(\sX ; {\Z}/n) \to KH(\sX; {\Z}/n)$ is a homotopy equivalence.
\item
For any integer $n$ nilpotent in $k$, the map
$f^*: \K(\sX ; {\Z}[{1}/n]) \to KH(\sX; {\Z}[1/n])$ is a homotopy equivalence.
\end{enumerate}
\end{cor}

\noindent\emph{Acknowledgements.} 
The authors would like to thank the referees for providing useful suggestions
on improving the contents and the presentation of this paper.

\end{document}